\documentclass[12pt]{article}
\usepackage{amsmath}
\usepackage{latexsym}
\usepackage{amssymb}
\usepackage{xcolor}
\newtheorem{theorem}{Theorem}[section]
\newtheorem{lemma}[theorem]{Lemma}

\newtheorem{Remark}[theorem]{Remark}
\newtheorem{Conj}[theorem]{Conjecture}
\newtheorem{proposition}[theorem]{Proposition}

\newtheorem{Example}[theorem]{Example}
\newtheorem{Number}[theorem]{\!\!}

\newtheorem{thmintro}{Theorem}

\newenvironment{remark}{\begin{Remark}\rm}{\end{Remark}}

\newenvironment{numba}{\begin{Number}\rm}{\end{Number}}
\newenvironment{proof}{{\noindent\bf Proof.}}%
                  {\nopagebreak\hspace*{\fill}$\Box$\medskip\par}

\newcommand{\at}{\symbol{'100}}

\newcommand{\mto}{\mapsto}
\newcommand{\N}{{\mathbb N}}
\newcommand{\bF}{{\mathbb F}}

\newcommand{\Z}{{\mathbb Z}}
\newcommand{\R}{{\mathbb R}}

\newcommand{\cA}{{\mathcal A}}

\newcommand{\cE}{{\mathcal E}}
\newcommand{\cF}{{\mathcal F}}

\DeclareMathOperator{\Aut}{Aut}
\DeclareMathOperator{\End}{End}
\newcommand{\sub}{\subseteq}
\DeclareMathOperator{\id}{id}

\DeclareMathOperator{\ad}{ad}
\DeclareMathOperator{\Hom}{Hom}
\def\Laurp{\mathbb{F}_p(\!(t)\!)}
\def\Powerp{\mathbb{F}_p[\![t]\!]}
\def\spacen{\mathbb{V}_n}
\def\spacej{\mathbb{V}_j}
\def\spacek{\mathbb{V}_k}
\def\spacejm{\mathbb{V}_{j-1}}
\def\IdA{I} 
\def\timest{[{\it t}\cdot]}

\def\timestd{[t^{\oplus d}\cdot]}
\def\timestc{[t^{\oplus d'}\cdot]}

\def\zz{\xi}
\def\yy{\eta}

\definecolor{satinsheengold}{RGB}{203, 161, 53}
\date{}
\title{ Locally pro-$p$ contraction groups are nilpotent}
\author{Helge Gl\"{o}ckner and George A. Willis\thanks{This research was supported by the Australian Research Council grant FL170100032
and the Deutsche Forschungsgemeinschaft, grant GL 357/10-1.}}
\begin{document}
\maketitle
\begin{abstract}
The authors have shown previously that every locally pro-$p$
contraction group decomposes into the direct product of a $p$-adic
analytic factor and a torsion factor. It has long been known that
$p$-adic analytic contraction groups are nilpotent. We show here
that the torsion factor is nilpotent too, and hence that every
locally pro-$p$ contraction group is nilpotent.  
\end{abstract}
{\bf Classification:} Primary 22D05;
secondary
20E22, 
20E36, 
20F18, 
20J06\\[3mm] 
{\bf Key words:} contraction group; torsion group; extension; 
abelian group; nilpotent group; isomorphism types\\[11mm]
\section{Introduction}
A \emph{contraction group} is a pair $(G,\tau)$, where~$G$ is a
locally compact group and~$\tau$ is an automorphism
of~$G$ such that $\tau^n(x)\to e$ as $n\to\infty$ for all~$x\in G$
(where~$e$ is the neutral element of $G$).
In~\cite{GW} the authors describe the structure of general totally
disconnected, locally compact contraction groups. In the present paper,
we answer a question which remained open in that description, namely,
whether locally pro-$p$ contraction groups must be nilpotent;
{\it cf.\/}~Problem~20.4.1 in \cite{CM}
(a topological group is \emph{locally pro-$p$} if it has an open subgroup
which is a pro-$p$ group.) The positive answer has implications
in the structure theory of totally disconnected, locally compact
groups ({\it cf.\/}~\cite{GWil}),
where contraction groups appear naturally in the study
of general automorphisms~\cite{BW}. It also has implications for Moufang twin trees, see~\cite{Zsystem} and Remark~\ref{rem:Zsystem}. 

Theorem~B in~\cite{GW} asserts that every totally disconnected
locally compact contraction~group is the direct product of a
torsion subgroup of finite exponent and a subgroup in which every
element is divisible. Furthermore, Theorem~3.3 in \cite{GW}
asserts that every totally disconnected locally compact
contraction group has a composition series in which the
composition factors are contraction groups having no closed normal
subgroups which are contraction-invariant.
In the case of locally pro-$p$ groups, these composition factors
are isomorphic to either a finite-dimensional vector space
over~$\mathbb{Q}_p$ or to~$\Laurp$ (see \cite[Theorem~A]{GW}).
Moreover, any locally pro-$p$ contraction group
is the direct product of a torsion subgroup and a $p$-adic Lie group
({\it cf.\/}~Theorem~B in \cite{GW}).
Locally pro-$p$ groups are therefore solvable and, moreover,
their $p$-adic factor is nilpotent (by \cite[Theorem 3.5 (ii)]{Wang}).
It is hence natural to ask whether the torsion factor is nilpotent as well. 

A torsion locally pro-$p$ contraction group $(G,\tau)$ has
a composition series in which each factor is isomorphic to~$\Laurp$,
and thus is a torsion group of exponent~$p$.
Hence there is a positive integer~$b$, at most equal to the
length of the composition series for~$(G,\tau)$, such that
$g^{p^b}=e$ for all~$g\in G$.
We will say that~$(G,\tau)$ is a \emph{$p$-power contraction group}.  
\begin{thmintro}
\label{thm:main}
Every $p$-power contraction group,~$(G,\tau)$, is nilpotent. 
\end{thmintro}

Throughout, $(\Laurp,+)$ will be the additive group of the field
of formal Laurent series over the finite field $\mathbb{F}_p$
(see, e.g., \cite{Wei}).
This will usually be abbreviated as $\Laurp$ for convenience.
Then $\Laurp$ is a totally disconnected, locally compact abelian
group and~$\Powerp$, the additive group of the ring of formal
power series, is a compact open subgroup. Elements of $\Laurp$
have the form $f =  \sum_{i\in\mathbb{Z}} f_it^i$, where
$f_i\in \mathbb{F}_p$ and
there is $I$ such that $f_i = 0$ for all~$i<I$. 

The map $ \timest : f\mapsto tf$ is an automorphism of
$\Laurp$ and is a contraction because  $\timest^kf\to 0$
as $k\to\infty$ for every $f\in\Laurp$. The automorphism $\timestd$ of $\Laurp^d$ is given by
$$
\timestd : (f_1,\dots, f_d)\mapsto (\timest(f_1), \dots, \timest(f_d))
= (tf_1,\dots, tf_d).
$$

We give
$\Aut(\Laurp^d)$ the Braconnier topology
({\it cf.\/}~\cite[Definition~9.14]{Str}).
Theorem~\ref{thm:main} follows from the next theorem,
which uses notation as in~(\ref{eq:adt}).
\begin{thmintro}
\label{thm:second_main}
Suppose that $\phi : \Laurp \to \Aut(\Laurp^d)$ is a continuous
homomorphism of groups and satisfies
\begin{equation}
\label{eq:phi_identity}
\phi\circ\timest  = \ad(\timestd)\circ \phi.
\end{equation}
Then there is $\zz\in \Laurp^d\setminus\{0\}$
such that $\phi(f)(\zz) = \zz$ for every $f\in \Laurp$.
\end{thmintro}

Theorem~\ref{thm:main} may be applied to give a new proof of a theorem about buildings shown in~\cite{Zsystem} and to answer a question about endomorphisms of profinite groups posed in~\cite{Reid}. Details are given below in Remarks~\ref{rem:Zsystem} and~\ref{rem:endo_profinite}.

The set of natural numbers, $\N$, is assumed throughout to include $0$.

\paragraph{Acknowledgements:}  This work was stimulated by and has benefitted from discussions with Pierre-Emmanuel Caprace, Bernhard M\"uhlherr and Colin Reid and it is a pleasure thank them for their contributions.

\section{Reduction of Theorem~\ref{thm:main} to Theorem~\ref{thm:second_main}}
\label{sec:reduction_to_2step}

Every totally disconnected locally compact contraction group $(G,\tau)$
has a composition series and, if the group is a $p$-power contraction group,
then all composition factors are isomorphic to $(\Laurp,\timest)$
({\it cf.\/}~Theorem 3.3 and Theorem~A in \cite{GW}).
The proof of Theorem~\ref{thm:main}
is by induction on the number of composition factors.
If there is one factor, then $G\cong \Laurp$ and $G$ is abelian;
likewise if $G=\{e\}$.

Assume that every $p$-power contraction group having at most~$k$
composition factors is nilpotent and consider a $p$-power
contraction group $(G,\tau)$ having~$k+1$ composition factors.
To show that~$G$ is nilpotent, it suffices to show that its centre
is non-trivial, for $Z(G)$ is a characteristic, and hence $\tau$-invariant,
subgroup of~$G$ and the quotient $(G/Z(G), \tau|^{G/Z(G)})$ is then a
contraction group having at most~$k$ composition factors. Hence $G/Z(G)$
is nilpotent, by the induction hypothesis, and~$G$ is nilpotent.

The second largest group in a composition
series for $G$ is a $\tau$-invariant normal subgroup,~$N$,
such that $(G/N,\tau|^{G/N})\cong (\Laurp,\timest)$ and~$(N,\tau|_N)$
is a contraction group having~$k$ composition factors.
The induction hypothesis implies that~$N$ is nilpotent and hence
that~$Z(N)$ is non-trivial. Since~$Z(N)$ is a characteristic subgroup
of~$N$, it is $\tau$-invariant. Moreover, the endomorphism
$z\mapsto pz : Z(N)\to Z(N)$ has non-trivial kernel,~$P$ say,
which is a characteristic subgroup of~$Z(G)$. Hence~$P$ is $\tau$-invariant
and the contraction group $(P,\tau|_P)$ is a torsion group of exponent~$p$.
Therefore $(P,\tau|_P) \cong (\Laurp^d, \timestd)$ for some $d>0$,
by \cite[Theorem~E]{GW2}.

Since~$P$ is a characteristic subgroup of~$N$ and~$N$ is normal in~$G$,
the assignment $\phi(g)(z):=gzg^{-1}$ defines a homomorphism
$\phi : G \to \Aut(P)$,
which is continuous as the map $G\times P\to P$, $(g,z)\mapsto
\phi(g)(z)$ is so (see \cite[Lemma~10.4 (c)]{Str}).
The proof of Theorem~\ref{thm:main} may be
completed by showing that there is a non-identity element $z\in P$
such that $\phi(g)(z)=z$ for all $g\in G$.

For $g\in G$ and $h\in P$ we have, by definition of~$\phi$, 
$$
\phi(\tau(g))(\tau(h)) = \tau(g)\tau(h)\tau(g)^{-1} = \tau(\phi(g)(h)).
$$
Hence $\phi(\tau(g)) = \tau|_P\circ \phi(g)\circ \tau^{-1}|_P$
for all~$g\in G$. 
Since~$P$ is contained in~$Z(N)$, $\phi$ induces a continuous homomorphism
$\phi|^{G/N} : G/N \to \Aut(P)$. Then we have
\begin{equation}
\label{eq:first_identity}
(\phi|^{G/N}\circ \tau|^{G/N})(g)
= \tau|_P\circ \phi|^{G/N}(g)\circ \tau|_P^{-1}\quad (g\in G/N).
\end{equation}

As previously remarked, $(G/N,\tau|^{G/N})\cong (\Laurp,\timest)$
and there is~ $d>0$ such that~$(P,\tau|_P) \cong (\Laurp^d, \timestd)$.
Under these isomorphisms, $\phi|^{G/N}$ corresponds to a continuous
homomorphism
$\Laurp \to \Aut(\Laurp^d)$ and it will be convenient to denote this
homomorphism too by $\phi$. It is also convenient to introduce the
notation $\ad(\timest)$ for the map 
\begin{equation}
\label{eq:adt}
\ad(\timest)(\psi)
= \timestd\circ \psi \circ \timestd^{-1}, \quad (\psi\in \Aut(\Laurp^d)).
\end{equation}
With these isomorphisms and notation, the identity~\eqref{eq:first_identity}
implies~\eqref{eq:phi_identity},\footnote{If $\phi\colon G/N\to\Aut(P)$
is \emph{any} continuous homomorphism and we write $\phi$ also for the corresponding
continuous homomorphism $\Laurp \to \Aut(\Laurp^d)$,
then the latter satisfies~\eqref{eq:phi_identity} if and only if the former
satisfies~\eqref{eq:first_identity}.}
and the proof of
Theorem~\ref{thm:main} may be completed by showing that
Theorem~\ref{thm:second_main} holds. 

\begin{remark}
Given a continuous homomorphism $\phi : \Laurp \to \Aut(\Laurp^d)$
satisfying~\eqref{eq:phi_identity}, form the semidirect product
$\Laurp^d\rtimes \Laurp$. Then the map 
$$
\timestd \times \timest : \Laurp^d\rtimes \Laurp\to \Laurp^d\rtimes \Laurp
$$ 
is a contractive automorphism. Since $\Laurp^d\rtimes \Laurp$ is nilpotent
only if there is $\zz\in \Laurp^d\setminus\{0\}$ such that
$\phi(g)(\zz) = \zz$ for every $g\in \Laurp$, Theorem~\ref{thm:second_main}
is necessary as well as sufficient for Theorem~\ref{thm:main}. 
\end{remark}

\section{The automorphism group of $\Laurp^d$}
\label{sec:Automorphism_group}

If $G$ is a topological abelian group, we let $\End(G)$ be the ring of all continuous
endomorphisms of~$G$.
The proof of Theorem~\ref{thm:second_main} goes by representing endomorphisms
in~$\End(\Laurp^d)$ as infinite block matrices with the blocks
being~$d\times d$ matrices over $\mathbb{F}_p$. In this section, conditions
on block matrices for them to correspond to elements of $\End(\Laurp^d)$
are determined, and then additional conditions for them to belong to the
image of $\phi : \Laurp \to \Aut(\Laurp^d)$ and
satisfy~\eqref{eq:phi_identity} are derived.

\begin{remark}
\label{examp:linear}
The group $\Laurp^d$ is a $d$-dimensional vector space over $\Laurp$
and every $\Laurp$-linear transformation belongs to~$\End(\Laurp^d)$.
However, unlike the characteristic~$0$ cases of~$(\mathbb{R}^d,+)$
and~$(\mathbb{Q}_p^d,+)$, in which every continuous group endomorphism is
linear over $\mathbb{R}$ and $\mathbb{Q}_p$ respectively, many continuous
endomorphisms of~$(\Laurp^d,+)$ are not linear over~$\Laurp$.  
\end{remark}

\subsubsection*{Representing endomorphisms by infinite matrices}

For each $n\in\mathbb{Z}$, let $\spacen =
\left\{ (a_1t^n,\dots,a_dt^n) \mid
a_r\in \mathbb{F}_p,\ r\in\{1,\dots,d\}\right\}$
be the subspace of $\Laurp^d$ consisting of $d$-tuples in which each
coordinate is a degree~$n$ monomial and let $\pi_n$ be the obvious
projection of $\Laurp^d$ onto~$\spacen$. Then $\spacen\cong \mathbb{F}_p^d$
and, given $i,j\in\mathbb{Z}$, each endomorphism
$A : \Laurp^d \to \Laurp^d$ determines an endomorphism
$\pi_i\circ A|_{\spacej} : \mathbb{F}_p^d \to \mathbb{F}_p^d$
via this isomorphism. Denote $\pi_i\circ A|_{\spacej}$ by $A_{i,j}$.
It will be seen that this representation of endomorphisms,~$A$, by
infinite block matrices, $(A_{i,j})_{i,j\in\mathbb{Z}}$, is faithful and
preserves the ring structure of~$\End(\Laurp^d)$.
We identify $A_{i,j}\in\Hom(\mathbb{V}_j,\mathbb{V}_i)$
and the correponding $(d\times d)$-matrix
with entries in~$\bF_p$; conversely,
any such matrix $A_{i,j}$ may be identified with
a homomorphism $\mathbb{V}_j\to\mathbb{V}_i$;
the meaning will be clear from the context.

That the matrix representation is faithful may be immediately verified.
\begin{proposition}
\label{prop:block_matrix}
For each endomorphism $A : \Laurp^d \to \Laurp^d$, the infinite block matrix $(A_{i,j})_{i,j\in\mathbb{Z}}$ satisfies: 
\begin{description}
\item[(M1)] for each $i\in\mathbb{Z}$ there is $J_i\in \mathbb{Z}$
such that $A_{i,j}=0$ for all $j>J_i$;
\label{prop:block_matrix1}
\item[(M2)] for each $j\in\mathbb{Z}$ there is $I_j\in\mathbb{Z}$
such that $A_{i,j}=0$ for all $i<I_j$; and
\label{prop:block_matrix2}
\item[(M3)] there are $I,J\in\mathbb{Z}$ such that $A_{i,j}=0$
whenever $i<I$ and $j>J$. 
\label{prop:block_matrix3}
\end{description}
\end{proposition}
\begin{proof}
{\bf(M1)} For each~$i\in\mathbb{Z}$, the map $\pi_i\circ A$ is a continuous
group homomorphism $(\Laurp^d,+) \to (\mathbb{F}_p^d,+)$.
Since $\mathbb{F}_p^d$ is discrete,
the kernel of $\pi_i\circ A$ is open and hence contains
$t^{J_i+1}\Powerp^d$ for some $J_i\in\mathbb{Z}$.
Thus $A_{i,j}=0$ for all $j>J_i$.

{\bf(M2)} For each $j\in\mathbb{Z}$, the image of the map
$A|_{\spacej}$ is a finite subset of
$\Laurp^d$ and every~$\zz$ in the image has $\zz_n=0$ eventually
as~$n\to-\infty$. 

{\bf(M3)} Continuity of~$A$ implies that there is $J\in\mathbb{Z}$
such that, if $\zz$ belongs to $t^J\Powerp^d$, then $A\zz$ is
in~$\Powerp^d$. Then $A_{i,j}=0$ for all $i<0$ and $j>J$.
\end{proof} 

If~$(A_{i,j})$ satisfies condition~{\bf(M1)} and~$(B_{i,j})$ satisfies
{\bf(M2)}, then the sum $\sum_{j\in\mathbb{Z}} A_{i,j}B_{j,k}$ is finite
for all $i,k$. While the product matrix~$(A_{i,j})(B_{i,j})$ need not
satisfy the same conditions, in the presence of~{\bf(M3)} these
conditions are preserved under matrix multiplication.
\begin{proposition}
\label{prop:MatricesM}
Denote the set of infinite ($d\times d$)-block matrices over $\mathbb{F}_p$
satisfying~{\bf(M1)}--{\bf(M3)} by $M_{d,\mathbb{Z}}(\mathbb{F}_p)$.
Then $M_{d,\mathbb{Z}}(\mathbb{F}_p)$ is a ring under the usual matrix addition
and multiplication and the map 
$$
A\mapsto (A_{i,j})_{i,j\in\mathbb{Z}} : \End(\Laurp^d)\to M_{d,\mathbb{Z}}(\mathbb{F}_p)
$$ 
is a ring homomorphism.
\end{proposition}
\begin{proof} That $M_{d,\mathbb{Z}}(\mathbb{F}_p)$ is closed under addition
may be easily verified and it has already been noted that the matrix product
of two elements of $M_{d,\mathbb{Z}}(\mathbb{F}_p)$ is defined.  

To see that the product of $(A_{i,j})$ and $(B_{i,j})$ in
$M_{d,\mathbb{Z}}(\mathbb{F}_p)$ satisfies~{\bf(M1)}, consider~$i\in\mathbb{Z}$.
Then there is $J_i$ such that $A_{i,j}=0$ for all $j>J_i$ and there are
$J_B,K_B$ such that $B_{j,k}=0$ whenever $j<J_B$ and $k>K_B$.
Moreover, since $(B_{i,j})$ satisfies~{\bf(M1)}, there is $K$ such that
$B_{j,k}=0$ for all $J_B\leq j\leq J_i$ and $k>K$. Then we have that
$B_{j,k}=0$ for all $j\leq J_i$ and $k>\max\{K,K_B\}$ and so
$\sum_{j\in\mathbb{Z}} A_{i,j}B_{j,k} = 0$ for all such~$k$. Since this holds
for all~$i\in\mathbb{Z}$, the product $(A_{i,j})(B_{i,j})$ satisfies~{\bf(M1)}. 
That $(A_{i,j})(B_{i,j})$ satisfies~{\bf(M3)} may be shown by essentially
the same argument. For this, choose~$I,J$ such that $A_{i,j}=0$ whenever
$i<I$ and $j>J$. Then choose $K$ such that $B_{j,k}=0$ for all
$J_B\leq j\leq J$ and $k>K$. Then $\sum_{j\in\mathbb{Z}} A_{i,j}B_{j,k} = 0$
for all $i<I$ and $k>\max\{K,K_B\}$. 
A similar argument shows that~{\bf(M2)} is also satisfied. 

That $M_{d,\mathbb{Z}}(\mathbb{F}_p)$ satisfies the ring axioms follows
in the usual way now that it has been shown that it is closed under
matrix addition and multiplication.

The map $A\mapsto (A_{i,j})_{i,j\in\mathbb{Z}}$ clearly preserves ring
addition. To see that it also preserves multiplication, note that for
any $A\in\End(\Laurp^d)$ we have
$$
A = \sum_{j\in\mathbb{Z}} A|_{\spacej}\circ \pi_j
$$
pointwise, since $z$ equals the limit
$\sum_{j\in\Z}\pi_j(z)$ of the net of finite partial sums
for each $z\in\Laurp^d$,
and $A$ is a continuous homomorphism.
Hence, for all $(i,k)\in\mathbb{Z}^2$,
\begin{align*}
\pi_i\circ (AB)|_{\spacek} &=
\sum_{j\in\mathbb{Z}}
\pi_i\circ\left(A|_{\spacej}\circ \pi_j\right)\circ B|_{\spacek},
\text{ because } (AB)|_{\spacek} = A\circ B|_{\spacek},\\
&= \sum_{j\in\mathbb{Z}}
\left(\pi_i\circ A|_{\spacej}\right)\circ \left(\pi_j\circ B|_{\spacek}\right),
\text{ by associativity.}
\end{align*}
$\;$\\[-14.5mm]
\end{proof} 
$\;$\\[-1.8mm]
Although not needed here, more can be shown:
\begin{proposition}
\label{prop-isomorphism}
The map $\End(\Laurp^d)\to M_{d,\mathbb{Z}}(\mathbb{F}_p)$,
$A\mapsto (A_{i,j})_{i,j\in\mathbb{Z}}$ is surjective and hence an isomorphism
of rings.
\end{proposition}
Furthermore, the Open
Mapping Theorem,~\cite[Theorem~5.29]{HaR}, implies that any endomorphism
of~$\Laurp^d$ which is a bijection belongs to $\Aut(\Laurp^d)$.\\[2.3mm]
Proposition~\ref{prop-isomorphism}
follows from the next lemma, which will be re-used later.
\begin{numba}
As a preparation, note that
$\timestd \Powerp^d$ has index $p^d$ in the additive group $\Powerp^d$.
More generally,
\[
[V\colon \tau(V)]=p^d
\]
for each compact open subgroup $V\subseteq \Laurp^d$
which is invariant under $\tau:=\timestd$,
as both $[V\colon \tau(V)]$ and $[\Powerp^d\colon \tau(\Powerp^d])=p^d$
equal the module $\Delta(\tau^{-1})$ of the automorphism~$\tau^{-1}$
of $\Laurp^d$ (see, e.g., \cite[Proposition 1.1(e)]{GW}
and its proof).
Considering $V$ as an $\mathbb{F}_p$-vector space,
we find an $\mathbb{F}_p$-vector subspace $F\subseteq V$
such that
\[
V=F \oplus \tau(V).
\]
Since $V/\tau(V)\cong F$ has $p^d$ elements, $F$ has dimension~$p$.
\end{numba}
\begin{numba}\label{summands-to-zero}
If $(z_n)_{n\in\N}$
is a sequence in $\Laurp^d$ such that $z_n\to 0$
as $n\to\infty$, then the series $\sum_{n=0}^\infty z_n$
converges in $\Laurp^d$ (see \cite[Proposition~23.1(ii)]{Sch}),
and so does the net of finite partial sums.
\end{numba}
\begin{lemma}
\label{lemma-make-endos}
Let $V$ be a $\tau$-invariant compact open subgroup
of $\Laurp^d$ and $b_1,\ldots, b_d$ be an $\mathbb{F}_p$-basis
for a vector space complement $F$ of $\tau(V)$ in $V$.
Abbreviate $\Phi:=\Z\times\{1,\ldots,d\}$.
Then the following holds:
\begin{itemize}
\item[\rm(a)]
The vectors
$\tau^j(b_k)$ are linearly independent for $(j,k)\in\Phi$,
and their $\mathbb{F}_p$-linear span
\begin{equation}\label{sumdirect}
\bigoplus_{j\in\Z}\tau^j(F)
\end{equation}
is dense in $\Laurp^d$.
\item[\rm(b)]
Each $z\in \Laurp^d$ can be written as a limit
\begin{equation}\label{elements-sums}
z=\sum_{(j,k)\in\Phi}n_{j,k}\, \tau^j(b_k)
\end{equation}
of finite partial sums with unique $n_{j,k}\in\mathbb{F}_p$ such that,
for some $j_0\in \Z$,
\[
n_{j,1}=\cdots=n_{j,d}=0\quad\mbox{for all $\,j<j_0$.}
\]
Given $m\in\Z$, we have $z\in \tau^m(V)$ if and only if
we can choose $j_0\geq m$.
\item[\rm(c)]
For all $(j,k)\in\Phi$, the map $\pi_{j,k}\colon
\Laurp^d\to\mathbb{F}_p$, $z\mapsto n_{j,k}$
is a continuous homomorphism.
\item[\rm(d)]
If $(w_n)_{n\in\Z}$ is a sequence in $\Laurp^d$ with
$w_n\to 0$ as $n\to\infty$, then
\[
\phi(z):=\sum_{(j,k)\in\Phi}n_{j,k}\,w_{k+jd}
\]
converges in $\Laurp^d$ for each $z\in\Laurp^d$ as in
{\rm\eqref{elements-sums}}, and defines an endomorphism
$\phi$ of the topological group~$\Laurp^d$.
\end{itemize}
\end{lemma}
\begin{proof}
(a) and (b): Let $j_0\in\Z$. For each integer $j\geq j_0$,
we have
\[
\tau^j(V)=\tau^j(F)\oplus \tau^{j+1}(V)
\]
and $\tau^j(b_1),\ldots,\tau^j(b_d)$ is an $\mathbb{F}_p$-basis
for $\tau^j(F)$. By induction,
\[
\tau^{j_0}(V)=\tau^{j_0}(F)\oplus\cdots\oplus\tau^j(F)\oplus \tau^{j+1}(V)
\]
for all $j\geq j_0$. As the sum $\tau^{j_0}(F)\oplus\cdots\oplus\tau^j(F)$
is direct for all $j,j_0\in\Z$ with $j\geq j_0$,
also the sum in \eqref{sumdirect}
is direct, and $(\tau^j(b_k))_{(j,k)\in\Phi}$
is a basis for the latter.
If $j_0\in \Z$ and $z\in\tau^{j_0}(V)$,
we recursively find $n_{j,1},\ldots,n_{j,d}\in\mathbb{F}_p$
for integers $j\geq j_0$ such
that
\[
z+\tau^{j+1}(V)=\sum_{i=j_0}^j\sum_{k=1}^dn_{i,k}\,\tau^i(b_k)+\tau^{j+1}(V).
\]
In fact, $z-\sum_{i=1}^j\sum_{k=1}^dn_{i,k}\,\tau^i(b_k)$
is in $\tau^{j+1}(V)$ and hence in
\[
\sum_{k=1}^dn_{j+1,k}\,\tau^{j+1}(b_k)+\tau^{j+2}(V)
\]
for suitable $n_{j+1,1},\ldots,n_{j+1,d}\in\mathbb{F}_p$,
as $\tau^{j+1}(V)=\tau^{j+1}(F)\oplus\tau^{j+2}(V)$
and $\tau^{j+1}(b_1),\ldots,\tau^{j+1}(b_d)$ is an $\mathbb{F}_p$-basis
for $\tau^{j+1}(F)$. Set $n_{j,k}:=0$ for $j<j_0$. Then
\[
z=\sum_{(j,k)\in\Phi}n_{j,k}\,\tau^j(b_k),
\]
as $z-\sum_{(j,k)\in I}n_{j,k}\tau^j(b_k)\in \tau^{j+1}(V)$ for each
finite subset $I\sub \Phi$ such that $\{j_0,\ldots,j\}\times\{1,\ldots,d\}\sub I$.
If also
\[
z=\sum_{(j,k)\in \Phi} m_{j,k}\,\tau^j(b_k)
\]
with zero coefficients for $j<j_1$, we let $J:=\min\{j_0,j_1\}$.
For $j\geq J$, we have
\[
z+\tau^{j+1}(V)=\sum_{i=J}^j\sum_{k=1}^dn_{j,k}\tau^j(b_k)+\tau^{j+1}(V)=
\sum_{i=J}^j\sum_{k=1}^d m_{j,k}\tau^j(b_k)+\tau^{j+1}(V)
\]
in $\tau^J(V)=\tau^J(F)\oplus \cdots\oplus \tau^j(F)\oplus \tau^{j+1}(V)$
and deduce that
$n_{j,k}=m_{j,k}$ for all $j\in\{J,\ldots,j\}$ and $k\in\{1,\ldots,d\}$.
The coefficients are thus unique. The representation \eqref{elements-sums}
implies the density of the direct sum \eqref{sumdirect} in $\Laurp^d$.
The construction shows that elements in $\tau^{j_0}(V)$
have coefficients $n_{j,k}=0$ for $j<j_0$.
Conversely, the sum \eqref{elements-sums} is in the closed $\tau$-invariant
subgroup $\tau^{j_0}(V)$ for such coefficients.

(c) It is easy to see that $\pi_{j,k}$ is a homomorphism of groups.
Therefore $\pi_{j,k}$ will be continuous if its restriction to $\tau^j(V)$
is continuous. But this map can be identified with the $k$th component
of the continuous quotient map
\[
\tau^j(V)\to \tau^j(V)/\tau^{j-1}(V)\cong \tau^j(F)
\cong \bigoplus_{k=1}^d\mathbb{F}_p\tau^j(b_j)\cong\mathbb{F}_p^d.
\]

(d) Let $\cF$ be the set of finite subsets of $\Phi$, directed via inclusion.
Given $z\in\Laurp^d$ and $j_0$ as in~\eqref{elements-sums},
write
\[
S_{z,I}:=
\sum_{(j,k)\in I}n_{j,k}\,w_{k+jd}
\]
for $I\in\cF$.
Given a compact open subgroup $Q\sub \Laurp^d$,
there exists $n_0\in\Z$ such that $w_n\in Q$ for all $n\geq n_0$.
There is $j_1\geq j_0$ such that $j_1d \geq n_0$.
Then
\[
S_{z,I}\in Q
\]
for all $I\in\cF$ such that $I\cap(\{j_0,\ldots,j_1\}\times \{1,\ldots, d\})=\emptyset$,
showing that $(S_{z,I})_{I\in\cF}$ is a Cauchy net in $\Laurp^d$
and thus convergent. We write $\phi(z)$ for the limit.
As each $S_{z,I}$ is additive in~$z$ and addition in $\Laurp^d$
is continuous, also $\phi(z)$ is additive in~$z$.
Thus $\phi$ is a homomorphism of groups.
To see that $\phi$ is continuous, it suffices to prove its
continuity on~$V$. There is $n_0\in \N$
such that $w_n\in \Powerp^d$ for all $n>n_0$.
There is an integer $\ell\leq 0$ such that
$w_0,\ldots,w_{n_0}\in \timestd^\ell\Powerp^d=:Q$.
Then~$Q$ is a compact open subgroup of $\Laurp^d$
such that $w_n\in Q$ for all $n\geq 0$,
and thus
\[
\phi(z)\in Q\quad\mbox{for all $\,z\in V$.}
\]
Since $Q$ is compact, the initial topology
on~$Q$ with respect to the
continuous homomorphisms $\pi_j|_Q$ for integers $j\geq \ell$,
which separate points, coincides with the given
topology on~$Q$. Hence $\phi|_V^Q$ will be continuous
if $\pi_j\circ \phi|_V$ is continuous
for each $j\geq \ell$. Given~$j$, there is $m_0\geq 0$ such that
$w_n\in \timestd^{j+1}\Powerp^d$ for all $n > m_0$ and thus
$\pi_j(w_n)=0$. There is an integer $M\geq 0$ such that
$(M+1)d\geq m_0$. Then
\[
(\pi_j\circ\phi)(z)=\sum_{m=0}^M\sum_{k=1}^d\pi_{m,k}(z)\,\pi_j(w_{k+md})
\]
for all $z\in V$, which is a continuous function of~$z\in V$.
\end{proof}
{\bf Proof of Proposition~\ref{prop-isomorphism}.}
As before, abbreviate $\tau:=\timestd$.
Let $(A_{i,j})_{i,j\in\Z}
\in M_{d,\mathbb{Z}}(\mathbb{F}_p)$.
Using Lemma~\ref{lemma-make-endos}
with $V:=\Powerp^d$, $F:=\mathbb{V}_0$,
the standard $\mathbb{F}_p$-basis
\begin{equation}\label{standard-basis}
e_1=(1,0,\ldots,0),\ldots, e_d=(0,\ldots,0,1)
\end{equation}
of $\mathbb{V}_0$ in place of $b_1,\ldots, b_d$
and
\[
w_{k+jd}:=\sum_{i\in\Z} A_{i,j}\tau^j(e_k)\in \Laurp^d
\]
for $(j,k)\in\Phi$
(which converges by (M2) and \ref{summands-to-zero}),
we get an endomorphism $\phi$ of $\Laurp^d$~which satisfies, for $j\in\Z$ and $z=
n_{j,1}\,\tau^j(e_1)+\cdots+ n_{j,d}\,\tau^j(e_d)\in\mathbb{V}_j$,
\[
\phi(z)
=\sum_{k=1}^d n_{j,k}\,w_{k+jd}
=\sum_{i\in\Z}A_{i,j}(z).
\]
Applying the continuous homomorphism $\pi_i$ for $i\in\Z$,
we obtain
$\pi_i(\phi(z))=A_{i,j}(z)$, whence $\pi_i\circ \phi|_{\mathbb{V}_j}=A_{i,j}$.
Thus $A\mapsto (A_{i,j})_{i,j\in\Z}$ is surjective.\hfill{\small$\square$}\\[2.3mm]
The next result
allows the basis for the matrix representation described
in Propositions~\ref{prop:block_matrix} and~\ref{prop:MatricesM}
to be changed.

\begin{proposition} 
\label{prop:change_basis} 
Suppose that $V$ is a compact open subgroup of $\Laurp^d$ such that
$\timestd(V) < V$. Then there is an automorphism, $\theta$, of $\Laurp^d$
such that $\theta\circ \timestd = \timestd\circ\, \theta$ and
$\theta(\Powerp^d) = V$.\end{proposition} 
%
\begin{proof}
Abbreviate $\tau:=\timestd$
and $\Phi:=\Z\times\{1,\ldots, d\}$.
Let $F\sub V$ be an $\mathbb{F}_p$-vector space such that
$V=F\oplus \tau(V)$
and $b_1,\ldots, b_d$
be an $\mathbb{F}_p$-basis for~$F$.
Using Lemma~\ref{lemma-make-endos}
with $w_{k+jd}:=\tau^j(e_k)$
for $(j,k)\in\Phi$
(where $e_1,\ldots, e_d\in\mathbb{V}_0$
is the standard basis),
we obtain an endomorphism $\phi$ of $\Laurp^d$
such that
\[
\phi(\tau^j(b_k))=\tau^j(e_k)\quad\mbox{for all $(j,k)\in\Phi$.}
\]
Using the lemma again with $\Powerp^d$
in place of~$V$, the complement $\mathbb{V}_0$ and its basis $e_1,\ldots, e_d$
in place of~$F$ and $b_1,\ldots, b_d$, respectively,
and using $w_{k+jd}:=\tau^j(b_k)$,
we obtain an endomorphism $\theta$ of $\Laurp^d$ such that
\[
\theta(\tau^j(e_k))=\tau^j(b_k)\quad\mbox{for all $(j,k)\in\Phi$.}
\]
Then $\theta\circ\phi$ is an endomorphism such that
\[
(\theta\circ\phi)(z)=z
\]
for all $z=\tau^j(b_k)$ with $(j,k)\in\Phi$.
As these elements generate a dense subgroup of $\Laurp^d$
by Lemma~\ref{lemma-make-endos}(a),
we deduce that $\theta\circ \phi=\id$.
Likewise, $\phi\circ\theta=\id$, whence $\theta$ is an automorphism
with $\theta^{-1}=\phi$.
Since
\[
\theta(\tau(z))=\theta(\tau^{j+1}(e_k))=\tau^{j+1}(b_k)=\tau(\theta(z))
\]
for all $z=\tau^j(e_k)$ with $(j,k)\in\Phi$,
arguing as before we get $\theta\circ \tau=\tau\circ \theta$.\\[2.3mm]
Since $\theta(\tau^j(e_k))=\tau^j(b_k)\in V$
for all $j\in\N$ and $k\in\{1,\ldots,d\}$,
we have $\theta(\Powerp^d)\sub V$.
As the image is compact,
it contains the closed subgroup of~$V$ generated by the $\tau^j(b_k)$
and hence coincides with~$V$.
\end{proof}

\subsubsection*{Homomorphisms from $\Laurp$ to $\Aut(\Laurp^d)$}

Let~$\phi : \Laurp\to \Aut(\Laurp^d)$ be a continuous homomorphism which satisfies~\eqref{eq:phi_identity}. Additional conditions satisfied by the $(d\times d)$-block matrices representing $\phi(f)$ for $f\in\Laurp$ are established now.

The next result identifies a subgroup,~$V$, such that, after the change of basis provided in Proposition~\ref{prop:change_basis}, the matrix representation of the image of~$\phi$ has a simple form.
We recall terminology.
\begin{numba}
Let $\tau$ be an automorphism of
a totally disconnected, locally compact group~$G$.
If $V$ is a compact open subgroup of~$G$, define
\[
V_+:=\bigcap_{n=0}^\infty \tau^n(V)\quad\mbox{and}\quad
V_-:=\bigcap_{n=0}^\infty\tau^{-n}(V).
\]
Then $\tau(V_-)\subseteq V_-$. The subgroup~$V$
is called \emph{tidy above} for~$\tau$ if
\[
V=V_+V_-
\]
({\it cf.\/}~\cite{GWil}).
If $\tau$ is a contractive automorphism,
$U\subseteq G$ an identity neighbourhood
and $K$ a compact subset of~$G$,
then there exists a positive integer~$n_0$
such that $\tau^n(K)\subseteq U$
for all $n\geq n_0$ (see
\cite[Lemma~1.4(iv)]{Sie}
or \cite[Proposition~2.1]{Wang}).
Hence $V_+=\{e\}$ in this case and thus $V$ is tidy above
if and only if $V=V_-$, which holds
if and only if $\tau(V)\subseteq V$.
\end{numba}
\begin{proposition}
\label{prop:invariant}
There is a compact open subgroup $V\leq \Laurp^d$ such that 
$$
\phi(f)(V) = V \mbox{ for every }f\in \Powerp\mbox{ and }\timestd(V) < V.
$$
\end{proposition}
\begin{proof}
Since $U := \Powerp^d$ is a compact and open subset of~$\Laurp^d$, 
$$
\mathcal{N} := \left\{ \alpha\in \Aut(\Laurp^d) \mid \alpha(U) = U\right\}
$$ 
is a neighbourhood of the identity in $\Aut(\Laurp^d)$ under the Braconnier topology. Since~$\phi$ is continuous, there is $m\geq0$ such that
$\phi(t^m\cdot\Powerp)\subseteq\mathcal{N}$ and then, since $\Powerp/(t^m\cdot\Powerp)$ is finite, 
$$
V := \bigcap \left\{ \phi(f)(U) \mid f\in \Powerp\right\}
$$
is a finite intersection. Then~$V$ is a compact open subgroup of $\Laurp^d$ and is invariant under $\phi(\Powerp)$. 

Since $\timest(\Powerp)\leq \Powerp$, the subgroup~$V$ is invariant under $\phi(\timest(\Powerp))$. Hence, since~$\phi$ satisfies~\eqref{eq:phi_identity}, $(\timestd)^{-1}(V)$ is invariant under~$\phi(\Powerp)$ and, iterating, $(\timestd)^{-k}(V)$ is invariant under~$\phi(\Powerp)$ for all $k\geq0$. By~\cite[Lemma~1]{GWil}, there is $K\geq0$ such that $\bigcap_{k=0}^K (\timestd)^{-k}(V)$ is tidy above for $\timestd$ and this intersection too is invariant under $\phi(\Powerp)$. Since $\timestd$ is a contraction on $\Laurp^d$, tidiness above for~$\timestd$ amounts to being invariant under~$\timestd$. Hence, replacing~$V$ with $\bigcap_{k=0}^K (\timestd)^{-k}(V)$, yields that~$V$ is invariant under~$\phi(\Powerp)$ and $\timestd(V)<V$. 
\end{proof}

Next, the condition that $\phi\circ\timest  = \ad(\timest)\circ \phi$ is translated to the matrix representation.
\begin{lemma}
\label{lem:matrix_condition}
For every~$f\in\Laurp$ and $(i,j)\in \mathbb{Z}^2$,
$$
\phi(\timest f)_{i,j} = \phi(f)_{i-1,j-1}.
$$ 
\end{lemma} 
\begin{proof}
By definition of~$\ad$ and because $\phi\circ\timest  = \ad(\timest)\circ \phi$, 
\begin{align*}
\pi_i\circ\phi(t\cdot f)|_{\spacej} &= \pi_i\circ (\timestd\circ(\phi(f)\circ \timestd^{-1})|_{\spacej}\\
&= \pi_{i-1}\circ \phi(f)|_{\spacejm}.
\end{align*}
$\;$\\[-13mm]
\end{proof}

Propositions~\ref{prop:invariant} and~\ref{prop:change_basis} imply that it may be assumed that $\Powerp^d$ is invariant under automorphisms in $\phi(\Powerp)$. Hence
\begin{equation}
\label{eq:entry_zero1}
\pi_i \circ \phi(f)|_{\spacej} = 0 \mbox{ for all }f\in \Powerp,\ i<0\mbox{ and }j\geq0.
\end{equation}
Moreover, if $f\in \Powerp$, then so is $t\cdot f$ and thus
Lemma~\ref{lem:matrix_condition} implies that
\begin{equation}
\label{eq:entry_zero2}
\pi_i \circ \phi(f)|_{\spacej} = 0 \mbox{ for all }f\in \Powerp,\ i<0\mbox{ and }j> i.
\end{equation}
Note that the evaluation map
\[
\Aut(\Laurp^d)\times\Laurp^d\to \Laurp^d,\quad
(\alpha,z)\mto \alpha(z)
\]
is continuous as the Braconnier topology is finer than the
compact-open topology and $\Laurp^d$ is locally compact
(see \cite[Lemma~9.8]{Str}).
Therefore, the map
\[
h_i\colon \Laurp \times \Laurp^d\to \mathbb{V}_i,\quad
(f,z)\mto \pi_i(\phi(f)(z))
\]
is continuous for all $i\in\Z$. Since $\mathbb{V}_i$ is
discrete, it follows that $h_i^{-1}(\{0\})$
is open in $\Laurp \times\Laurp^d$.
Since $\alpha(0)=0$ for all $\alpha\in\Aut(\Laurp^d)$, we have
\[
\Powerp \times \{0\}\sub h_i^{-1}(\{0\}).
\]
As both factors on the left-hand side are compact,
using the Wallace Theorem (see \cite{Kel})
we find an open $0$-neighbourhood
$U\sub \Laurp^d$ such that $\Powerp\times U\sub h_i^{-1}(\{0\})$.
There is $J_i\in\Z$ such that $\timestd^{J_i}\Powerp^d\sub U$.
Then
\begin{equation}\label{simultaneously}
\pi_i\circ \phi(f)|_{\mathbb{V}_j}=0
\mbox{ for all $f\in\Powerp$ and $j\geq J_i$.}
\end{equation}
Improving on~\eqref{eq:entry_zero2}, it will be seen next that, when $f \in \Powerp$ and $i<0$, it occurs that~$\pi_i \circ \phi(f)|_{\spacej} = 0$ for some values of $j<i$. 

For this, note that
\[
Q_m:=\{\alpha\in \Aut(\Laurp^d)\colon (\alpha-\IdA)(\Powerp^d)\subseteq \timestd^m(\Powerp^d)\}
\]
is an identity neighbourhood in $\Aut(\Laurp^d)$ for each $m\in \Z$,
as the Braconnier topology on $\Aut(\Laurp^d)$
is finer than the compact-open topology, which coincides with the
topology induced
on $\Aut(\Laurp^d)$ by the topology of compact convergence
on $(C(\Laurp^d,\Laurp^d),+)$. Since $\phi$ is continuous and $\phi(0)=\IdA$,
the pre-image $\phi^{-1}(Q_m)$ is a $0$-neighbourhood in $\Laurp$.
We therefore find $N_m\in \Z$ such that,
for every $g\in t^{N_m}\Powerp$,
\begin{equation}\label{pre-step}
\phi(g)(\zz) \in \zz + (\timestd)^m(\Powerp^d) \mbox{ for every }\zz\in \Powerp^d.
\end{equation}
For all $f\in \Powerp$ and $n\geq N_m$, we can apply \eqref{pre-step}
to $g:=t^nf$.
Rearranging with the help of \eqref{eq:adt}, condition
\eqref{pre-step} implies that
$$
(\phi(f) - \IdA)((\timestd)^{-n}\Powerp^d) \leq (\timestd)^{m-n}(\Powerp^d) \mbox{ for all }n\geq N_m.
$$
Hence $\pi_i\circ (\phi(f) - \IdA)|_{\spacej} = 0$ for all $j\geq -n$ and $i< m-n$.
Put $a_m := m - N_m$. Then every $i<a_m$ is less than $m-n$ with $n = m-i-1$ (which
is $\geq N_m$) and so 
\begin{equation}
\label{eq:entry_zero3}
\pi_i\circ (\phi(f)-\IdA)|_{\spacej} = 0\mbox{ for every }j\geq 1+i-m. 
\end{equation}
Equations
\eqref{simultaneously},
\eqref{eq:entry_zero2}, and \eqref{eq:entry_zero3}
can be summarized as follows.
\begin{proposition}
\label{prop:matrix_entries_zero}
\hspace*{-1.3mm}Assume\hspace*{-.1mm} that $\hspace*{-.1mm}\Powerp^d\hspace*{-.6mm}$ is
\hspace*{-.1mm}invariant \hspace*{-.1mm}under \hspace*{-.1mm}the
\hspace*{-.1mm}subgroup~$\hspace*{-.1mm}\phi(\Powerp)$ of $\Aut(\Laurp^d)$. Then there are sequences $(a_m)_{m\in \mathbb{N}}$ with $a_0=0$ (which may be assumed non-increasing), and $(b_i)_{i\in\mathbb{N}}$ (which may be assumed to be non-decreasing)
such that,
for all $f\in\Powerp$,
$$
\pi_i\circ (\phi(f) - \IdA)|_{\spacej} = 0 \,\mbox{ for all }i,\,j\in \mathbb{Z} \mbox{ with } 
\begin{cases} 
j> i-m, & \mbox{ if }\,i< a_m\\
j> i + b_i, & \mbox{ if }\,i\geq0
\end{cases}.\vspace{-5.5mm}
$$
\endproof
\end{proposition}
%

\section{Proof of Theorem~\ref{thm:second_main}}
\label{sec:proof_of_B}

To show, as claimed in Theorem~\ref{thm:second_main}, that there is $\zz$ with $\phi(f)(\zz) = \zz$ for every $f\in \Laurp$, it suffices to show that $\phi(t^r)(\zz) = \zz$ for every $r\in\mathbb{Z}$ because $\langle t^r \mid r\in\mathbb{Z}\rangle$ is dense in $\Laurp$.
This will be achieved in three steps. Denote by $A = (A_{i,j})$ the $(d\times d)$-block matrix with $A_{i,j} = \pi_i\circ (\phi(t^0)-\IdA)|_{\spacej}$. Then the first step is to show that $A$ may be assumed to be lower triangular; the second that~$A$ may be assumed to be strictly lower triangular; and the third deals with the strictly lower triangular case.  

By the following approximation lemma, it will suffice to solve
the above fixed point problem for finite sets of powers of~$t$.

\subsubsection*{An approximation lemma}
\label{sec:approximation}

In the following, we let $(\Powerp^d)'$
be the dual space of all continuous $\bF_p$-linear
functionals $\lambda\colon \Powerp^d\to\bF_p$.
\begin{lemma}
\label{lem:approximation}
Let $(\mathcal{E}_n)_{n\in \mathbb{N}}$ be an increasing sequence of subsets of
$(\Powerp^{d})'$
and suppose that, for each $n\in \mathbb{N}$, there is $\yy^{(n)}\in \Powerp^{d}$
with $\yy^{(n)}_0\ne0$ and 
\begin{equation*}
\label{eq:approximation1}
\lambda(\yy^{(n)})=0
\mbox{ for all }\lambda\in \mathcal{E}_n.
\end{equation*}
Then there is $\yy\in\Powerp^{d}$ such that $\yy_0\ne 0$ and
$\lambda(\yy)=0$ for all $\lambda \in \bigcup_{n\in\N}\cE_n$.
\end{lemma}

\begin{proof}
As $\Powerp^d$ is compact and metrizable,
the sequence $(\yy^{(n)})_{n\in\N}$
has a convergent subsequence $(\yy^{(n_k)})_{k\in\N}$.
Let $\yy\in\Powerp^d$ be its limit.
For $n\in\N$ and $\lambda\in\cE_n$,
we have $\lambda(\yy^{(n_k)})=0$
for all $k\in\N$ such that $n_k\geq n$, and thus
\[
\lambda(\yy)=\lim_{k\to\infty}\lambda(\yy^{(n_k)})=0.
\]
Hence $\lambda(\yy)=0$ for all $\lambda\in\bigcup_{n\in\N}\cE_n$.
Since $\yy^{(n_k)}_0\to \yy_0$ in $\bF_p^d$
and $\bF_p^d\setminus\{0\}$ is closed, we have $\yy_0\not=0$.
\end{proof}

\subsubsection*{Above the diagonal}
\label{sec:above_diagonal}

The first step, showing that $A$ may be assumed to be lower triangular, is accomplished by passing to a
subrepresentation
$f\mto \phi(f)|_{\mathcal{U}}$
for a $\phi(\Laurp)$-invariant and $\timestd$-invariant
closed subgroup $\mathcal{U}$ of $\Laurp^d$.
Towards finding this subgroup, define 
\begin{equation}
\label{eq:define_U0}
\mathcal{U}_0 = \left\{ \yy\in \Laurp^d \mid \phi(f)\yy\in\Powerp^d \mbox{ for all }f\in \Laurp\right\}.
\end{equation}
Then $\mathcal{U}_0$ is a closed subgroup of~$\Laurp^d$ and $\mathcal{U}_0 \leq \Powerp^d$ because $\phi(0) = \IdA$. It will be seen later that $\mathcal{U}_0$ is non-zero.

Its definition implies that~$\mathcal{U}_0$ is invariant under $\phi(\Laurp)$, because $\phi(\Laurp)$ is a subgroup of $\Aut(\Laurp^d)$. It is also invariant under~$\timestd$ because, for $\yy\in\mathcal{U}_0$ and~$f\in \Laurp$, we have
\begin{align}
\label{eq:U0_invariant_under_t}
\phi(f)\timestd\yy &= \timestd\left(\ad(\timestd)^{-1}(\phi(f))\right)\yy \\
&= \timestd \phi(\timest^{-1} f) \yy \leq \timestd\Powerp^d.\notag
\end{align}
Define $\mathcal{U}_n := \timestd^n\mathcal{U}_0$
for $n\in\mathbb{Z}$. Then $\mathcal{U}_m\leq \mathcal{U}_n$ if $m\geq n$ and the calculation in~\eqref{eq:U0_invariant_under_t} shows that
$$
\mathcal{U}_n = \left\{ \yy\in \Laurp^d \mid \phi(f)\yy\leq \timestd^n\Powerp^d \mbox{ for all }f\in \Laurp\right\}.
$$
Define $\mathcal{U} = \bigcup_{n\in\mathbb{Z}} \mathcal{U}_n$.
By \cite[Proposition~3.1(a)]{GW},
$\mathcal{U}$ is a $\timestd$-invariant closed subgroup of~$\Laurp^d$
which has $\mathcal{U}_0$ as an open subgroup;
notably, $(\mathcal{U},\timestd|_{\mathcal{U}})$
is a contraction group.
Also, $\mathcal{U}$ is $\phi(\Laurp)$-invariant as each $\mathcal{U}_n$
is so, and
\[
\mathcal{U}\cong\Laurp^{d'}
\]
as a contraction group for some $d'\leq d$
(as a composition series of $(\mathcal{U},\timestd|_{\mathcal{U}})$
can be filled up to a composition series of $\Laurp^d$).

Proposition~\ref{prop:change_basis} implies that there are: an isomorphism $\theta : \mathcal{U} \to \Laurp^{d'}$, such that $\theta(\mathcal{U}_0) = \Powerp^{d'}$ and
$\theta\circ \timestd|_{\mathcal{U}} = \timestc\circ\, \theta$
and a homomorphism $\psi : \Laurp\to \Aut(\Powerp^{d'})$, such that $\theta \circ \phi(f)|_{\mathcal{U}} = \psi(f)\circ \theta$ for every $f \in \Laurp$. Then we have, for every $n\in\mathbb{Z}$, 
$$
\psi(f)\timestc^n\Powerp^{d'}\leq \timestc^n\Powerp^{d'}
$$ 
and the $(d'\times d')$-block matrix $\left(\pi_i\circ \psi(f)|_{\spacej}\right)_{(i,j)\in \mathbb{Z}^2}$ is thus lower triangular for every $f\in\Laurp$. Passing to the subgroup $\mathcal{U}$, it suffices, therefore, to prove Theorem~\ref{thm:second_main} in the case when $\phi$ maps $\Laurp$ to lower triangular matrices. 

It remains to show that the subgroup $\mathcal{U}_0$ defined in~\eqref{eq:define_U0} is not trivial.
To do so, it suffices to find non-zero $\yy\in \Powerp^d$ such that $\phi(t^{-n}f)\yy \in \Powerp^d$ for all $f\in \Powerp^d$ and $n\in\mathbb{Z}$, which is equivalent to
\begin{equation}
\label{eq:n_condition}
A(f)\timestd^n\yy\in \timestd^n\Powerp^d
\end{equation}
for all $n\in\mathbb{Z}$,
where $A(f) = (A(f)_{i,j})_{i,j\in\mathbb{Z}}$ with $A(f)_{i,j}
:= (\phi(f)-\IdA)_{i,j}=\pi_i\circ(\phi(f)-\IdA)|_{\spacej}$. Denote the matrix $\ad(\timestd^r)A(f)$
by $A(f)^{(r)}$ and its entries by $A(f)^{(r)}_{i,j}$.
Then $A(f)^{(r)}_{i,j} = A(f)_{i-r,j-r} = \pi_i\circ(\phi(t^rf)-\IdA)|_{\spacej}$ and,
since $\phi(f)$ and $\phi(t^rg)$ commute, $A(f)$ and~$A(g)^{(r)}$
commute for all $f,g\in \Laurp$.
If $A(f)$ is already lower triangular for all $f\in\Powerp$
(and thus for all $f\in \Laurp$),
there is nothing to prove (then
$\Powerp^d\sub \mathcal{U}$).
We may suppose therefore that there are integers $i^*$ and~$j^*$ with $j^*>i^*$ such that $A(f)_{i^*,j^*} \ne 0$
for some $f\in \Powerp$,
and that $i^*$ is the smallest integer for which that is the case.
Having fixed $i^*$, we also choose $j^*$ maximal.
Then $i^*\geq0$ by Proposition~\ref{prop:matrix_entries_zero}. 

Fix a positive integer $c$. We begin by finding a non-zero $\yy\in\Laurp^d$ such that \eqref{eq:n_condition} holds for all $n\leq c$. To this end, recall from Proposition~\ref{prop:matrix_entries_zero} that there is an integer $a$ such that
\begin{equation}
\label{eq:c_restriction}
A(f)_{i,j} = 0 \mbox{ if }i<a \mbox{ and }j\geq i-c,
\end{equation}
for all $f\in \Powerp$.
Then $a\leq i^*$ and we choose $a$ to be the largest integer for
which~\eqref{eq:c_restriction} holds. In that case, there is~$b$ with 
\begin{equation}
\label{eq:choose_b}
b \geq a-c \mbox{ and } A(f^*)_{a,b}\ne0
\end{equation}
for some $f^*\in \Powerp$,
and we choose $b$ to be the largest integer for which~\eqref{eq:choose_b} holds.
Having maximized~$b$, we fix the element $f^*\in \Powerp$ with
$A(f^*)_{a,b}\ne 0$ (it may be that $a=i^*$ and $b=j^*$). Then we have that 
 \begin{equation}
 \label{eq:Azerocases}
 A(f)_{i,j} = 0 \mbox{ if }
\begin{cases}
i<a\mbox{ and }j\geq i-c\\
i=a\mbox{ and }j>b\\
i<i^*\mbox{ and }j> i\\
i=i^*\mbox{ and }j>j^*
\end{cases}
\end{equation}
for all $f\in\Powerp$.

Suppose that $m$ and $n$ are integers with $n>m\geq i^*$. Then 
\begin{align}
\label{eq:first_sum}
\left(A(f)A(f^*)^{(n-a)}\right)_{m,b+n-a} &= \sum_{r\in\mathbb{Z}} A(f)_{m,r}
A(f^*)_{r+a-n,b}\notag\\
&= \sum_{r\geq n} A(f)_{m,r} A(f^*)_{r+a-n,b}, 
\end{align}
because, if $r<n$, then $r+a-n<a$ and $b> (r+a-n)-c$ by the choice of~$b$ in~\eqref{eq:choose_b}, and~\eqref{eq:Azerocases} may be applied to conclude that
$A(f^*)_{r+a-n,b}=0$. On the other hand,
\begin{align}
\label{eq:second_sum} 
\left(A(f^*)^{(n-a)}A(f)\right)_{m,b+n-a} &= \sum_{r\in\mathbb{Z}} A(f^*)_{m+a-n,r+a-n}
A(f)_{r,b+n-a}\notag\\
&= \sum_{r<m-c} A(f^*)_{m+a-n,r+a-n}A(f)_{r,b+n-a}
\end{align}
because $m+a-n < a$ and, if $r\geq m-c$, then~$r+a-n \geq (m+a-n)-c$ and it follows by~\eqref{eq:Azerocases} that $A(f^*)_{m+a-n,r+a-n} = 0$. 

If we now suppose that $m\leq i^*+c$, then the remaining terms in~\eqref{eq:second_sum} vanish as well, because $r<m-c$ implies that~$r<i^*$ while~\eqref{eq:choose_b} implies that $b+n-a \geq n-c$ and we have that~$n-c>m-c>r$. Hence it follows from~\eqref{eq:Azerocases} that~$A(f)_{r,b+n-a}=0$. 

Since $A(f)$ and~$A(f^*)^{(n-a)}$ commute,~\eqref{eq:first_sum},~\eqref{eq:second_sum} and the fact that~\eqref{eq:second_sum} equals zero imply that
$$
\sum_{r\geq n} A(f)_{m,r} A(f^*)_{r+a-n,b} = 0 \mbox{ for all }m\leq i^*+c\mbox{ and }n>m.
$$
The $r=n$ term of this series is $A(f)_{m,n}A(f^*)_{a,b}$.
Since~$A(f^*)_{a,b}\ne0$, by~\eqref{eq:choose_b}, we may choose
$z\in \mathbb{F}_p^d$ such that $A(f^*)_{a,b}
z\ne0$. Define $\bar{z} = zt^b\in\Laurp^d$ and $\bar{\yy} =
(\timestd)^{-a}A(f^*)\bar{z}$. Then: 
\begin{itemize}
\item $\bar{\yy}\in \Powerp^d$ with $\bar{\yy}_0$ non-zero; and 
\item $(A(f)\timestd^{n}\bar{\yy})_m = 0$ for every $f\in\Powerp$,
$m\leq i^*+c$ and $n>m$. 
\end{itemize}
In particular,
\begin{equation}\label{herefunctionals}
(A(f)\timestd^{n}\bar{\yy})_m = 0 \mbox{\;for all $f\in \Powerp$,
$n\leq c$ and $m<n$,}
\end{equation}
which implies that \eqref{eq:n_condition} holds for all $n\leq c$.
Since \eqref{herefunctionals}
holds for every $c>0$ with some $\bar{\yy}\ne 0$,
Lemma~\ref{lem:approximation} implies\footnote{The set $\cE_c$
consists of the components of the maps $\Powerp^d \to \bF_p^d$,
$\zeta\mto (A(f)\timestd^n\zeta)_m$ with $n\leq c$, $m<n$,
and $f\in\Powerp^d$.}
that there is $\yy\in \Powerp^d$ with $\yy_0\ne0$ such that, for all $c>0$,
\[
(A(f)\timestd^{n}\yy)_m = 0 \mbox{\;for all $f\in \Powerp$,
$n\leq c$ and $m<n$,}
\]
so that
\eqref{eq:n_condition} holds for every $n$. Thus
$\yy\in \mathcal{U}_0$ and $\mathcal{U}_0$ is therefore non-zero.

\subsubsection*{On the diagonal}
\label{sec:on_diagonal}

It is seen below that the rest of the argument falls into two subcases. The following result will be used to treat one of these subcases. 
In the following, we let
$\mathcal{A}$ be the subalgebra of $\End(\Laurp^d)$
generated by $\{\phi(t^r)-\IdA\}_{r\in \mathbb{Z}}$.
For positive integers~$N$, we abbreviate
\[
\cA^N:=\{a_1\cdots a_N\colon a_1,\ldots,a_N\in\cA\}.
\]
\begin{lemma}
\label{lem:on_diagonal}
Suppose that the $(d\times d)$ block matrix $\phi(f)$
is lower triangular for every $f\in \Laurp$.
Then $B_{i,i} =0$ for every $B\in \mathcal{A}^d$. 
\end{lemma}
\begin{proof}
Consider an eigenvalue, $\lambda$, of $\phi(t^r)_{i,i}$ in some extension of $\mathbb{F}_p$. Then $\lambda^{p^k-1} = 1$, where $p^k$ is the order of the extension. On the other hand, the order of $\phi(t^r)_{i,i}$ is~$1$ or~$p$ because $pt^r = 0$ and it follows that $\lambda^p=1$. Therefore $\lambda=1$ is an eigenvalue of $\phi(t^r)_{i,i}$ for every $i\in\mathbb{Z}$ and is the only eigenvalue in any extension of $\mathbb{F}_p$. Hence $\phi(t^r)_{i,i} - \IdA_d$ is nilpotent for all $i,r\in\mathbb{Z}$. 

Since $\phi(t^0)$ and $\phi(t^r)$ commute and $\phi(t^r)_{i,i} = \phi(t^0)_{i-r,i-r}$ for all $i,r\in\mathbb{Z}$, it follows that $\{\phi(t^0)_{i,i}\}_{i\in\mathbb{Z}}$ is a commuting set and, moreover, that this set of equal to $\left\{ \phi(t^r)_{i,i}\mid i,r\in\mathbb{Z}\right\}$. Then $\left\{ \phi(t^r)_{i,i} - \IdA_d\mid i,r\in\mathbb{Z}\right\}$ is a commuting set of nilpotent $d\times d$ matrices over $\mathbb{F}_p$,
and thus generates a commutative $\bF_p$-algebra of matrices.
By Engel's Theorem (see Theorem~3.2
in \cite[Part~I, Chapter~V]{Ser}),
after a change of basis all algebra elements can be turned
into strict upper triangular matrices,
entailing that the product of any~$d$ elements in this algebra
is equal to~$0$. As~$B_{i,i}$ is such a product, the claim follows.
 \end{proof}  
 
  \subsubsection*{Below the diagonal}
\label{sec:below_diagonal}

For the remainder of the argument it is assumed that $\phi(t^0)$ is a lower triangular $(d\times d)$ block matrix. As before,
let~$\mathcal{A}$ be the subalgebra of $\End(\Laurp^d)$ generated by $\{\phi(t^r)-\IdA\}_{r\in \mathbb{Z}}$. The argument falls into two cases: when~$\mathcal{A}$ is nilpotent and when it is not. Since the elements $t^n$, $n\in\mathbb{Z}$, generate a dense subgroup of $\Laurp$, to prove Theorem~\ref{thm:second_main} it suffices to show that there is $\zz\in \Laurp^d$ such that $\phi(t^r)(\zz) = \zz$ for every $r\in \mathbb{Z}$.

\begin{proposition}
\label{prop:A_nilpotent}
Suppose there is $N\in \mathbb{N}$ such that $\mathcal{A}^N = \{0\}{}$. Then there is $\zz\in \Laurp^d\setminus\{0\}$ such that $\phi(t^n)(\zz) = \zz$ for every $n\in \mathbb{Z}$.
\end{proposition}  
\begin{proof}
If $N=1$, then $\mathcal{A} = \{0\}$ and $\phi(t^n) = \{\IdA\}$ for every $n\in \mathbb{Z}$ and so $\phi(t^n)(\zz) = \zz$ for every $\zz\in \Laurp^d$. 

Suppose that $N$ is the smallest integer with $\mathcal{A}^N = \{0\}$ and that
$N>1$.
Then $\mathcal{A}^{N-1} \ne \{0\}$ and we may choose $B\in \mathcal{A}^{N-1}\setminus\{0\}$. There is then $\yy\in \Laurp^d$ such that $\zz := B\yy \ne 0$ and $\mathcal{A}\zz =\{0\}$. This~$\zz$ satisfies the claim because $\phi(t^n) - \IdA\in \mathcal{A}$ for every $n\in\mathbb{Z}$.
\end{proof}

\begin{proposition} 
\label{prop:A_not_nilpotent}
Suppose that $\mathcal{A}^N \ne \{0\}{}$ for every $N\in \mathbb{N}$. Then there is $\zz\in \Laurp^d\setminus\{0\}$ such that $\phi(t^n)(\zz) = \zz$ for every $n\in \mathbb{Z}$. 
\end{proposition}
\begin{proof}
When $N\geq d$, Lemma~\ref{lem:on_diagonal} implies that $(\mathcal{A}^N)_{i,i} = \{0\}$ for every $i\in \mathbb{Z}$, that is, $\mathcal{A}^N$ is strictly lower triangular for $N\geq d$. For all $q\geq 1$,
it follows that the matrices in $\mathcal{A}^N$
are zero on the first~$q-1$ sub-diagonals for all $N\geq qd$. 
Also note that, for $N\geq 1$,
not all matrices in $\cA^N$
are block diagonal, because $\cA^{dN}=\{0\}$
otherwise by Lemma~\ref{lem:on_diagonal},
contrary to the current hypothesis.

Recall that $A = (A_{i,j})_{i,j\in\mathbb{Z}}$ with $A_{i,j} = \pi_i\circ(\phi(t^0)-\IdA)|_{\spacej}$ and that $A^{(n)} = \ad(\timestd^n)(A)$ satisfies that $A^{(n)}_{i,j} = A_{i-n,j-n} = \pi_i\circ(\phi(t^n)-\IdA)|_{\spacej}$. It will be shown that there is $\zz\in \Powerp^d$ with $\zz_0\ne0$ such that $(A\circ\timestd^{-n})(\zz)_m = 0$ for all $m,n\in\mathbb{Z}$. Then $A^{(n)}(\zz)$ is equal to zero for every~$n\in\mathbb{Z}$ and~$\zz$ is the claimed element of~$\Laurp^d$.

For any lower triangular block matrix,~$B$, and for~$j\leq i\in \mathbb{Z}$ and $q>0$, we have
\begin{align}
\label{eq:first_below}
(AB^{(j)})_{i,j-q} &= \sum_{k\in\mathbb{Z}} A_{i,k}B^{(j)}_{k,j-q} \notag\\
&= \sum_{j-q\leq k\leq i} A_{i,k}B_{k-j,-q}.
\end{align}
Multiplying~$A$ and~$B^{(j)}$ in the opposite order yields 
\begin{align}
\label{eq:second_below}
(B^{(j)}A)_{i,j-q} &= \sum_{k\in\mathbb{Z}} B^{(j)}_{i,k}A_{k,j-q} \notag\\
&= \sum_{j-q\leq k\leq i} B_{i-j,k-j}A_{k,j-q}.
\end{align}
Next, let $r>0$ be fixed and consider $-r< i < r$ and $i-r< j\leq i$. By Proposition~\ref{prop:matrix_entries_zero}, there is $a_r>0$ such that $A_{i,j}=0$
if $i < -a_r$ and $j>i-r$. Then, by hypothesis, we may choose $q\geq r+a_r$
and $B\in \mathcal{A}^N$ with
$N := (r+a_r)d$ such that $B_{0,-q} \ne 0$ and $B_{i,j} =0$ for all
$i,j\in \mathbb{Z}$ with $j-i>-q$. With these choices of~$q$ and~$B$, we have $B_{k-j,-q} = 0$ for all $k<j$ and~\eqref{eq:first_below} reduces to
$$
(AB^{(j)})_{i,j-q} = \sum_{j\leq k\leq i} A_{i,k}B_{k-j,-q}.
$$
We also have that $B_{i-j,k-j} = 0$ if $k> i-q$ and
$A_{k,j-q} = 0$ if $k\leq i-q$,
using that $k\leq i-q<r-q\leq -a_r$ in this case
and $j-q\geq j + k -i>k-r$. Hence
\eqref{eq:second_below} reduces to
$$
(B^{(j)}A)_{i,j-q} = 0.
$$
Since~$A$ and $B^{(j)}$ commute, it follows that, for all $-r< i < r$ and $i-r< j\leq i$,
\[
(AB^{(j)})_{i,j-q}=(B^{(j)}A)_{i,j-q}=0.
\]
As $B_{0,-q} \ne 0$, we may choose $z\in \mathbb{F}_p^d$ such that $B_{0,-q}z\ne 0$. Setting $\bar{z} = zt^{-q}\in \Laurp^d$ and $\yy = B\bar{z}$, we achieve that
\begin{itemize}
\item $\yy\in \Powerp^d$ and $\yy_0\in \mathbb{F}_p^d$ is non-zero; and 
\item $(A\timestd^{j}\yy)_i = 0$ for every $-r<i<r$ and $j>i-r$. 
\end{itemize}
Since this holds for every $r>0$, it follows by Lemma~\ref{lem:approximation} that there is $\zz\in \Powerp^d$ such that $\zz_0 \ne 0$ and $(A\timestd^{j}\zz)_i = 0$ for all $i,j\in \mathbb{Z}$.
\end{proof}

\begin{remark}
The `on the diagonal' step is the only point in the above argument where it is essential that the field $\Laurp$ has positive characteristic.
The claim may fail otherwise
because, for example, the representation
of~$\R$ given by the exponential map $t\mapsto \exp(t) : \mathbb{R} \to \mathbb{R}^\times\cong
\Aut(\mathbb{R})$ has no joint fixed vectors
and $\mathbb{R} \rtimes \mathbb{R}^\times$ is
a solvable group but not nilpotent.
\end{remark}

\begin{remark}
When the scale of $\tau^{-1}$ is equal to $p^2$ and the composition series for $(G,\tau)$ has length~$2$, the proof of Theorem~\ref{thm:main} is simpler while still illustrating the main steps in the proof of Theorem~\ref{thm:second_main}. Suppose that the composition series is
$$
\{e\} = G_0 < G_1 < G_2 = G,
$$
so that $G_1 \cong \Laurp \cong G/G_1$. Then showing that $G$ is nilpotent reduces to showing that $G_1$ is contained in the centre of $G$, which is equivalent to showing that $\phi(f)$ is the identity automorphism of $\Laurp$ for all $f\in \Laurp$. Hence the proof in this case is a verification that the ($(1\times1)$-block) matrix of $A =\phi(t^0)$ is the identity, rather than a proof of existence of an eigenvector with eigenvalue~$1$. For the verification, the arguments in the proof of Theorem~\ref{thm:second_main} show in turn that: all $a_{ij}$  above the diagonal are $0$ (by contradiction); all $a_{ij}$ on the diagonal equal~$1$ (because $A$ has order~$p$ but $(\mathbb{F}_p\setminus\{0\},\times)$ has order $p-1$); and all $a_{ij}$ below the diagonal are $0$ (by contradiction).
\end{remark}

\begin{remark}
\label{rem:Zsystem}
Theorem~A in~\cite{Zsystem} may be deduced from Theorem~\ref{thm:main} in the present paper. Although the former paper concerns Moufang twin trees, it is not necessary to know what they are in order to follow the proof because~\cite[Theorem~A]{Zsystem} is deduced from another purely group-theoretic result. To do this, the notion of a \emph{$\mathbb{Z}$-system of order $p$} is defined. This is a group, $X$, and a sequence $(x_n)_{n\in\mathbb{Z}}\subset X$ satisfying certain axioms set out in~\cite[Definition~3.2]{Zsystem}. Then~\cite[Theorem~A]{Zsystem} is deduced from~\cite[Theorem~3.4]{Zsystem}, which asserts that if $(X,(x_n)_{n\in\mathbb{Z}})$ is a $\mathbb{Z}$-system of prime order, then $X$ is nilpotent of class at most~$2$. 

That result follows from our Theorem~\ref{thm:main}. Briefly, let $(X,(x_n)_{n\in\mathbb{Z}})$ be a $\mathbb{Z}$-system of order $p$, and let $t$ be the automorphism of $X$ such that $t(x_n) = x_{n+2}$. Then $X_{0,\infty} = \langle x_k \mid k\geq0\rangle$, see~\cite[Definition~4.1]{Zsystem}, is a commensurated subgroup of $X\rtimes\langle t\rangle$. Denote the Schlichting completion (see~\cite{ReidWes,Tzanev}) of $X\rtimes\langle t\rangle$ relative to $X_{0,\infty}$ by $\widetilde{X}\rtimes\langle t\rangle$ and let $\widetilde{X}_{0,\infty}$ be the closure of $X_{0,\infty}$ in this completion. Then $\widetilde{X}_{0,\infty}$ is a pro-$p$ open subgroup of $\widetilde{X}\rtimes\langle t\rangle$ and $t(\widetilde{X}_{0,\infty})<\widetilde{X}_{0,\infty}$ with $[\widetilde{X}_{0,\infty} : t(\widetilde{X}_{0,\infty})] = p^2$. Hence the pair $(\widetilde{X}, t)$ is a contraction group having composition series of length~$2$ and composition factors isomorphic to $\Laurp$, and so $\widetilde{X}$ is nilpotent of class at most~$2$ by Theorem~\ref{thm:main}. Since $X$ is dense in $\widetilde{X}$, it too is nilpotent of class at most~$2$.
\end{remark}
\begin{remark}
\label{rem:endo_profinite}
The following question is posed in~\cite{Reid}: suppose that $G$ is a topologically finitely generated profinite group and that $\phi $ is an open self-embedding of $G$, is $\sf{con}(\phi)$ nilpotent? Here, 
$$
\sf{con}(\phi) = \left\{ g\in G \mid \phi^n(g)\to \id\mbox{as }n\to\infty\right\}
$$ 
and is closed, by~\cite[Theorem~A]{Reid}, and residually nilpotent, by~\cite[Corollary~1.3]{Reid}. Since $\sf{con}(\phi)$ is compact, it is in fact a projective limit of nilpotent groups and hence \emph{pro-nilpotent} in the sense of~\cite{Wsn}. We now answer the question.

Following~\cite{GW2}, $G$ is called
\emph{locally pro-nilpotent} if
$G$ has a pro-nilpotent open subgroup.
For each locally pro-nilpotent contraction group
$(G,\tau)$, there exist a finite set of primes~$P$
and $\tau$-invariant closed normal subgroups $G_p$ of~$G$ for $p\in P$
which are locally pro-$p$, such that
\[
G\cong \prod_{p\in P}G_p
\]
internally as a topological group (see \cite[Proposition~2.16]{GW2}).
Thus Theorem~\ref{thm:main} in the present paper entails
that every locally pro-nilpotent contraction group is nilpotent and hence, in answer to Reid's question in~\cite{Reid}, that $\sf{con}(\phi)$ is indeed nilpotent.
\end{remark}
{\small{\bf Helge  Gl\"{o}ckner}, Institut f\"{u}r Mathematik, Universit\"at Paderborn,\\
Warburger Str.\ 100, 33098 Paderborn, Germany,
email: {\tt  glockner\at{}math.upb.de};\\[1mm]
\emph{also conjoint professor at} Department of Mathematics,
University of Newcastle, Callaghan, NSW 2308, Australia.\\[3mm]
{\bf George A. Willis}, Department of Mathematics, University of Newcastle,\\
Callaghan, NSW 2308, Australia, email: {\tt George.Willis\at{}newcastle.edu.au}}\vfill
\end{document}